\documentclass[numbers=enddot,12pt,final,onecolumn,notitlepage]{scrartcl}%
\usepackage[headsepline,footsepline,manualmark]{scrlayer-scrpage}
\usepackage{amsfonts}
\usepackage{amssymb}
\usepackage{amsmath}
\usepackage{amsthm}
\usepackage{framed}
\usepackage{comment}
\usepackage{color}
\usepackage[breaklinks=true]{hyperref}
\usepackage[sc]{mathpazo}
\usepackage[T1]{fontenc}
\usepackage{needspace}
\usepackage{tabls}
\providecommand{\U}[1]{\protect\rule{.1in}{.1in}}
\theoremstyle{definition}
\newtheorem{theo}{Theorem}[section]
\newenvironment{theorem}[1][]
{\begin{theo}[#1]\begin{leftbar}}
{\end{leftbar}\end{theo}}
\newtheorem{lem}[theo]{Lemma}
\newenvironment{lemma}[1][]
{\begin{lem}[#1]\begin{leftbar}}
{\end{leftbar}\end{lem}}
\newtheorem{prop}[theo]{Proposition}

\newtheorem{defi}[theo]{Definition}

\newtheorem{remk}[theo]{Remark}
\newenvironment{remark}[1][]
{\begin{remk}[#1]\begin{leftbar}}
{\end{leftbar}\end{remk}}
\newtheorem{coro}[theo]{Corollary}
\newenvironment{corollary}[1][]
{\begin{coro}[#1]\begin{leftbar}}
{\end{leftbar}\end{coro}}
\newtheorem{conv}[theo]{Convention}

\newtheorem{quest}[theo]{Question}
\newenvironment{question}[1][]
{\begin{quest}[#1]\begin{leftbar}}
{\end{leftbar}\end{quest}}
\newtheorem{warn}[theo]{Warning}

\newtheorem{conj}[theo]{Conjecture}

\newtheorem{exam}[theo]{Example}
\newenvironment{example}[1][]
{\begin{exam}[#1]\begin{leftbar}}
{\end{leftbar}\end{exam}}

\let\sumnonlimits\sum
\let\prodnonlimits\prod
\let\cupnonlimits\bigcup
\let\capnonlimits\bigcap
\renewcommand{\sum}{\sumnonlimits\limits}
\renewcommand{\prod}{\prodnonlimits\limits}
\renewcommand{\bigcup}{\cupnonlimits\limits}
\renewcommand{\bigcap}{\capnonlimits\limits}
\ihead{The diagonal derivative of a skew Schur polynomial}
\ohead{page \thepage}
\begin{document}

\title{The diagonal derivative of a skew Schur polynomial}
\author{Darij Grinberg\thanks{Department of Mathematics, Drexel University,
Philadelphia, U.S.A.
(\href{mailto:darijgrinberg@gmail.com}{\texttt{darijgrinberg@gmail.com}})},
Nazar Korniichuk\thanks{Kyiv Natural-Scientific Lyceum No. 145, Kyiv, Ukraine
(\href{mailto:n.korniychuk.a@gmail.com}{\texttt{n.korniychuk.a@gmail.com}})},\\Kostiantyn Molokanov\thanks{Kyiv Natural-Scientific Lyceum No. 145, Kyiv,
Ukraine
(\href{mailto:kostyamolokanov@gmail.com}{\texttt{kostyamolokanov@gmail.com}}%
)}, and Severyn Khomych\thanks{Brucknergymnasium Wels, Austria
(\href{mailto:severyn.khomych@gmail.com}{\texttt{severyn.khomych@gmail.com}})}}
\date{version 1.0, 21 February 2024}
\maketitle

\begin{abstract}
\textbf{Abstract.} We prove a formula for the image of a skew Schur polynomial
$s_{\lambda/\mu}\left(  x_{1}, x_{2}, \ldots, x_{N}\right)  $ under the
differential operator $\nabla:= \dfrac{\partial}{\partial x_{1}}%
+\dfrac{\partial}{\partial x_{2}}+\cdots+\dfrac{\partial}{\partial x_{N}}$.
This generalizes a formula of Weigandt for $\nabla\left(  s_{\lambda}\right)
$.

\end{abstract}

\section{Notations and definitions}

Fix a nonnegative integer $N$. Let $R = \mathbb{Z}\left[  x_{1},x_{2}%
,\ldots,x_{N}\right]  $ be the ring of polynomials in $N$ indeterminates
$x_{1},x_{2},\ldots,x_{N}$ with integer coefficients.

Let $\nabla:R\rightarrow R$ be the operator $\dfrac{\partial}{\partial x_{1}%
}+\dfrac{\partial}{\partial x_{2}}+\cdots+\dfrac{\partial}{\partial x_{N}}$.
This map $\nabla$ is a derivation (i.e., it is $\mathbb{Z}$-linear and
satisfies $\nabla\left(  fg\right)  =\left(  \nabla f\right)  g+f\left(
\nabla g\right)  $ for all $f,g\in R$). We call it the \emph{diagonal
derivative} since (in the language of analysis) it is the directional
derivative with respect to the vector $\left(  1,1,\ldots,1\right)  $. (The
notation $\nabla$ comes from the related operator in \cite{Nenashev}, but this
is not the vector differential operator $\nabla$ known from analysis.)

We let $\left[  N \right]  $ denote the set $\left\{  1, 2, \ldots, N
\right\}  $.

For each $i\in\left[  N\right]  $, we let $e_{i}$ be the $N$-tuple $\left(
0,0,\ldots,0,1,0,0,\ldots,0\right)  \in\mathbb{Z}^{N}$ where the $1$ is at the
$i$-th position. Addition and subtraction of $N$-tuples are defined entrywise
(i.e., these $N$-tuples are viewed as vectors in the $\mathbb{Z}$-module
$\mathbb{Z}^{N}$). Thus, if $\mu\in\mathbb{Z}^{N}$ is any $N$-tuple, then
$\mu+e_{i}$ is the $N$-tuple obtained from $\mu$ by increasing the $i$-th
entry by $1$.

We shall use the standard notations regarding symmetric polynomials in $N$
variables $x_{1}, x_{2}, \ldots, x_{N}$ as introduced (e.g.) in
\cite{Stembridge} (but we write $N$ for what was called $n$ in
\cite{Stembridge}).

If $a\in\mathbb{Z}^{N}$ is any $N$-tuple, and if $i\in\left[  N\right]  $,
then the notation $a_{i}$ shall denote the $i$-th entry of $a$ (so that
$a=\left(  a_{1},a_{2},\ldots,a_{N}\right)  $).

We let $\mathcal{P}_{N}$ denote the set of all $N$-tuples $a\in\mathbb{Z}^{N}$
that satisfy
\[
a_{1}\geq a_{2}\geq\cdots\geq a_{N}\geq0.
\]
For instance, if $N=3$, then the $N$-tuples $\left(  3,1,0\right)  $ and
$\left(  2,2,1\right)  $ belong to $\mathcal{P}_{N}$, while the $N$-tuples
$\left(  2,1,-1\right)  $ and $\left(  1,2,1\right)  $ don't.

The $N$-tuples in $\mathcal{P}_{N}$ are called \emph{partitions of length
$\leq N$} (or \emph{partitions with at most $N$ nonzero terms}, following the
terminology of \cite{Stembridge}). In algebraic combinatorics, such an
$N$-tuple $a\in\mathcal{P}_{N}$ usually gets identified with the infinite
sequence $\left(  a_{1},a_{2},\ldots,a_{N},0,0,0,\ldots\right)  $, which is
called an (integer) partition. Any integer partition has a so-called
\emph{Young diagram} assigned to it (see, e.g., \cite[\S 1.7]{EC1}). If
$\mu\in\mathcal{P}_{N}$, and if $i\in\left[  N\right]  $ is such that the
$N$-tuple $\mu+e_{i}$ again belongs to $\mathcal{P}_{N}$, then the Young
diagram of $\mu+e_{i}$ is obtained from the Young diagram of $\mu$ by adding a
single box.

For any integer $n$, we let $h_{n}$ denote the complete homogeneous symmetric
polynomial in the variables $x_{1},x_{2},\ldots,x_{N}$. It is defined by
\begin{equation}
h_{n}:=\sum_{\substack{\left(  i_{1},i_{2},\ldots,i_{N}\right)  \in
\mathbb{N}^{N};\\i_{1}+i_{2}+\cdots+i_{N}=n}}x_{1}^{i_{1}}x_{2}^{i_{2}}\cdots
x_{N}^{i_{N}}, \label{eq.def.hn}%
\end{equation}
where $\mathbb{N}:=\left\{  0,1,2,\ldots\right\}  $ is the set of all
nonnegative integers (so that the sum ranges over all $N$-tuples $\left(
i_{1},i_{2},\ldots,i_{N}\right)  $ of nonnegative integers satisfying
$i_{1}+i_{2}+\cdots+i_{N}=n$). (Thus, $h_{0}=1$, and $h_{n}=0$ for each
negative $n$.) Clearly, $h_{n}\in R$ for each $n\in\mathbb{Z}$.

If $\lambda$ and $\mu$ are two $N$-tuples in $\mathcal{P}_{N}$, then the
notation $s_{\lambda/\mu}$ denotes the skew Schur polynomial $s_{\lambda/\mu
}\left(  x_{1},x_{2},\ldots,x_{N}\right)  $ (see, e.g., \cite{Stembridge} for
a definition\footnote{To be more precise, $s_{\lambda/\mu}$ is defined in
\cite{Stembridge} in the case when $\mu\leq\lambda$ (meaning that $\mu_{i}%
\leq\lambda_{i}$ for each $i\in\left[  N\right]  $). In all other cases,
$s_{\lambda/\mu}$ is defined to be $0$.}). (Note that this is called a
\textquotedblleft skew Schur function\textquotedblright\ in \cite{Stembridge},
but more commonly the latter word is reserved for the analogous object in
infinitely many variables.) The \emph{Jacobi--Trudi formula} says that any
$\lambda\in\mathcal{P}_{N}$ and $\mu\in\mathcal{P}_{N}$ satisfy
\begin{equation}
s_{\lambda/\mu}=\det\left(  h_{\lambda_{i}-\mu_{j}-i+j}\right)  _{i,j\in
\left[  N\right]  } \label{pf.thm.1.JT}%
\end{equation}
(where the notation $\left(  a_{i,j}\right)  _{i,j\in\left[  N\right]  }$
means the $N\times N$-matrix with given entries $a_{i,j}$). Proofs of this
formula can be found (e.g.) in \cite[(2.4.16)]{GriRei} or \cite[(7.69)]%
{EC2}\footnote{Both texts \cite[(2.4.16)]{GriRei} and \cite[(7.69)]{EC2} state
the Jacobi--Trudi formula for Schur functions in infinitely many variables
$x_{1},x_{2},x_{3},\ldots$ instead of Schur polynomials in finitely many
variables $x_{1},x_{2},\ldots,x_{N}$. However, the latter version can be
obtained from the former by setting $x_{N+1},x_{N+2},x_{N+3},\ldots$ to $0$.
\par
Note also that \cite[(7.69)]{EC2} assumes that $\mu\subseteq\lambda$, but the
proofs do not require this assumption.}. We can use the formula
\eqref{pf.thm.1.JT} as a definition of $s_{\lambda/\mu}$ here, as we will need
no other properties of $s_{\lambda/\mu}$. However, it is worth observing that
\begin{equation}
s_{\lambda/\mu}=0 \label{pf.thm.1.s=0}%
\end{equation}
unless $\mu\subseteq\lambda$ (that is, unless $\mu_{i}\leq\lambda_{i}$ for
each $i\in\left[  N\right]  $).


\section{The results}

The main result of this note is the following formula, which generalizes
Weigandt's recent result \cite[The Symmetric Derivative Rule]{Weigan23}:

\begin{theorem}
\label{thm.1} Let $a$ and $b$ be two integers such that $a+b=N-1$. Let
$\lambda\in\mathcal{P}_{N}$ and $\mu\in\mathcal{P}_{N}$. Define two further
$N$-tuples $\ell\in\mathbb{Z}^{N}$ and $m\in\mathbb{Z}^{N}$ by setting%
\[
\ell_{i}:=\lambda_{i}-i\ \ \ \ \ \ \ \ \ \ \text{and}\ \ \ \ \ \ \ \ \ \ m_{i}%
:=\mu_{i}-i\ \ \ \ \ \ \ \ \ \ \text{for each }i\in\left[  N\right]  .
\]
Then,
\[
\nabla\left(  s_{\lambda/\mu}\right)  =\sum_{\substack{i\in\left[  N\right]
;\\\lambda-e_{i}\in\mathcal{P}_{N}}}\left(  \ell_{i}+a\right)  s_{\left(
\lambda-e_{i}\right)  /\mu}+\sum_{\substack{i\in\left[  N\right]  ;\\\mu
+e_{i}\in\mathcal{P}_{N}}}\left(  b-m_{i}\right)  s_{\lambda/\left(  \mu
+e_{i}\right)  }.
\]

\end{theorem}

\begin{example}
Let $N=3$ and $\lambda=\left(  3,2,1\right)  $ and $\mu=\left(  1,1,0\right)
$. Then, the $N$-tuples $\ell$ and $m$ defined in Theorem \ref{thm.1} are
$\ell=\left(  2,0,-2\right)  $ and $m=\left(  0,-1,-3\right)  $. Let $a$ and
$b$ be two integers such that $a+b=N-1=2$. Thus, Theorem \ref{thm.1} says that%
\begin{align*}
&  \nabla\left(  s_{\left(  3,2,1\right)  /\left(  1,1,0\right)  }\right) \\
&  =\sum_{i\in\left\{  1,2,3\right\}  }\left(  \ell_{i}+a\right)  s_{\left(
\left(  3,2,1\right)  -e_{i}\right)  /\left(  1,1,0\right)  }+\sum
_{i\in\left\{  1,3\right\}  }\left(  b-m_{i}\right)  s_{\left(  3,2,1\right)
/\left(  \left(  1,1,0\right)  +e_{i}\right)  }\\
&  =\left(  2+a\right)  s_{\left(  2,1,1\right)  /\left(  1,1,0\right)
}+\left(  0+a\right)  s_{\left(  3,1,1\right)  /\left(  1,1,0\right)
}+\left(  -2+a\right)  s_{\left(  3,2,0\right)  /\left(  1,1,0\right)  }\\
&  \ \ \ \ \ \ \ \ \ \ +\left(  b-0\right)  s_{\left(  3,2,1\right)  /\left(
2,1,0\right)  }+\left(  b-\left(  -3\right)  \right)  s_{\left(  3,2,1\right)
/\left(  1,1,1\right)  }.
\end{align*}
Note that the second sum has no $i=2$ addend, since $\mu+e_{2}=\left(
1,2,0\right)  \notin\mathcal{P}_{N}$.
\end{example}

\begin{remark}
Let us comment on the combinatorial meaning of Theorem~\ref{thm.1}. A pair
$\left(  \mu,\lambda\right)  $ of $N$-tuples $\lambda,\mu\in\mathcal{P}_{N}$
is called a \emph{skew partition} if $\mu\subseteq\lambda$ (that is, $\mu
_{i}\leq\lambda_{i}$ for each $i\in\left[  N\right]  $). Such a skew partition
$\left(  \mu,\lambda\right)  $ has a \emph{skew Young diagram} assigned to it,
which is defined as the set difference of the Young diagrams of $\lambda$ and
$\mu$. This diagram is denoted by $\lambda/\mu$.

Let $\lambda/\mu$ be a skew partition with $\lambda,\mu\in\mathcal{P}_{N}$.
The sum $\sum_{\substack{i\in\left[  N\right]  ;\\\lambda-e_{i}\in
\mathcal{P}_{N}}}\left(  \ell_{i}+a\right)  s_{\left(  \lambda-e_{i}\right)
/\mu}$ in Theorem~\ref{thm.1} can be rewritten as $\sum_{\substack{i\in\left[
N\right]  ;\\\lambda-e_{i}\in\mathcal{P}_{N};\\\mu\subseteq\lambda-e_{i}%
}}\left(  \ell_{i}+a\right)  s_{\left(  \lambda-e_{i}\right)  /\mu}$, because
any addend in which $\mu\subseteq\lambda-e_{i}$ does not hold is $0$ by
(\ref{pf.thm.1.s=0}). This is a sum over all skew partitions obtained from
$\lambda/\mu$ by removing an outer corner (i.e., a removable box on the
southeastern boundary of $\lambda/\mu$). Moreover, the number $\ell
_{i}=\lambda_{i}-i$ tells us which diagonal this corner belongs to.

Likewise, the sum $\sum_{\substack{i\in\left[  N\right]  ;\\\mu+e_{i}%
\in\mathcal{P}_{N}}}\left(  b-m_{i}\right)  s_{\lambda/\left(  \mu
+e_{i}\right)  }$ can be rewritten as $\sum_{\substack{i\in\left[  N\right]
;\\\mu+e_{i}\in\mathcal{P}_{N};\\\mu+e_{i}\subseteq\lambda}}\left(
b-m_{i}\right)  s_{\lambda/\left(  \mu+e_{i}\right)  }$, which is a sum over
all skew partitions obtained from $\lambda/\mu$ by removing an inner corner
(i.e., a removable box on the northwestern boundary of $\lambda/\mu$).
Moreover, the number $m_{i}=\mu_{i}-i$ tells us which diagonal this corner
belongs to.
\end{remark}

\begin{remark}
Let $\lambda\in\mathcal{P}_{N}$. Let $\ell_{i}:=\lambda_{i}-i$ for each
$i\in\left[  N\right]  $. Set $\mu:=\left(  0,0,\ldots,0\right)
\in\mathcal{P}_{N}$. Then, the skew Schur polynomial $s_{\lambda/\mu}$ is
usually called $s_{\lambda}$. The only $i\in\left[  N\right]  $ satisfying
$\mu+e_{i}\in\mathcal{P}_{N}$ is $1$. Hence, Theorem~\ref{thm.1} (applied to
$a=N$ and $b=-1$) yields
\begin{align*}
\nabla\left(  s_{\lambda}\right)   &  =\sum_{\substack{i\in\left[  N\right]
;\\\lambda-e_{i}\in\mathcal{P}_{N}}}\left(  \ell_{i}+N\right)  s_{\lambda
-e_{i}}+\underbrace{\left(  1-0+\left(  -1\right)  \right)  }_{=0}%
s_{\lambda/\left(  \mu+e_{1}\right)  }\\
&  =\sum_{\substack{i\in\left[  N\right]  ;\\\lambda-e_{i}\in\mathcal{P}_{N}%
}}\left(  \ell_{i}+N\right)  s_{\lambda-e_{i}},
\end{align*}
which recovers \cite[The Symmetric Derivative Rule]{Weigan23}.
\end{remark}

Once Theorem \ref{thm.1} is proved, we will derive the following curious corollary:

\begin{corollary}
\label{cor.2} Let $\lambda\in\mathcal{P}_{N}$ and $\mu\in\mathcal{P}_{N}$.
Then,
\[
\sum_{\substack{i\in\left[  N\right]  ;\\\lambda-e_{i}\in\mathcal{P}_{N}%
}}s_{\left(  \lambda-e_{i}\right)  /\mu}=\sum_{\substack{i\in\left[  N\right]
;\\\mu+e_{i}\in\mathcal{P}_{N}}}s_{\lambda/\left(  \mu+e_{i}\right)  }.
\]

\end{corollary}

We note that Corollary \ref{cor.2} follows easily from the theory of skewing
operators (see \cite[\S 2.8]{GriRei}) and the Pieri rule.\footnote{In a
nutshell: Corollary \ref{cor.2} is obtained by taking the equality $s_{\mu
}^{\perp}\left(  s_{1}^{\perp}s_{\lambda}\right)  =\left(  s_{1}s_{\mu
}\right)  ^{\perp}s_{\lambda}$, which holds on the level of symmetric
functions in infinitely many indeterminates $x_{1},x_{2},x_{3},\ldots$ (a
consequence of \cite[Proposition 2.8.2 (ii)]{GriRei}), and expanding both of
its sides using the Pieri rules \cite[(2.7.1) and (2.8.3)]{GriRei}.} But we
shall prove both Theorem \ref{thm.1} and Corollary \ref{cor.2} elementarily.

\section{A lemma on $\nabla\left(  h_{n}\right)  $}

First we need a formula for the image of the complete homogeneous symmetric
polynomial $h_{n}$ under the operator $\nabla$:

\begin{lemma}
\label{lem.Nabla-h} Let $n$ be a integer. Then, $\nabla\left(  h_{n}\right)
=\left(  n+N-1\right)  h_{n-1}$.
\end{lemma}

\begin{proof}
[Proof of Lemma \ref{lem.Nabla-h}.]For each $k\in\left[  N\right]  $, we have
\begin{align*}
\dfrac{\partial}{\partial x_{k}}h_{n}  &  =\dfrac{\partial}{\partial x_{k}%
}\sum_{\substack{\left(  i_{1},i_{2},\ldots,i_{N}\right)  \in\mathbb{N}%
^{N};\\i_{1}+i_{2}+\cdots+i_{N}=n}}x_{1}^{i_{1}}x_{2}^{i_{2}}\cdots
x_{N}^{i_{N}}\ \ \ \ \ \ \ \ \ \ \left(  \text{by \eqref{eq.def.hn}}\right) \\
&  =\sum_{\substack{\left(  i_{1},i_{2},\ldots,i_{N}\right)  \in\mathbb{N}%
^{N};\\i_{1}+i_{2}+\cdots+i_{N}=n}}\underbrace{\dfrac{\partial}{\partial
x_{k}}\left(  x_{1}^{i_{1}}x_{2}^{i_{2}}\cdots x_{N}^{i_{N}}\right)
}_{\substack{=%
\begin{cases}
i_{k}x_{1}^{i_{1}}x_{2}^{i_{2}}\cdots x_{k}^{i_{k}-1}\cdots x_{N}^{i_{N}}, &
\text{if }i_{k}>0;\\
0, & \text{if }i_{k}=0
\end{cases}
\\\text{(where }x_{1}^{i_{1}}x_{2}^{i_{2}}\cdots x_{k}^{i_{k}-1}\cdots
x_{N}^{i_{N}}\text{ means the monomial }x_{1}^{i_{1}}x_{2}^{i_{2}}\cdots
x_{N}^{i_{N}}\\\text{with the exponent on }x_{k}\text{ decremented by
}1\text{)}}}\\
&  =\sum_{\substack{\left(  i_{1},i_{2},\ldots,i_{N}\right)  \in\mathbb{N}%
^{N};\\i_{1}+i_{2}+\cdots+i_{N}=n;\\i_{k}>0}}i_{k}x_{1}^{i_{1}}x_{2}^{i_{2}%
}\cdots x_{k}^{i_{k}-1}\cdots x_{N}^{i_{N}}\\
&  =\underbrace{\sum_{\substack{\left(  i_{1},i_{2},\ldots,i_{N}\right)
\in\mathbb{N}^{N};\\i_{1}+i_{2}+\cdots+\left(  i_{k}+1\right)  +\cdots
+i_{N}=n}}}_{=\sum_{\substack{\left(  i_{1},i_{2},\ldots,i_{N}\right)
\in\mathbb{N}^{N};\\i_{1}+i_{2}+\cdots+i_{N}=n-1}}}\left(  i_{k}+1\right)
\underbrace{x_{1}^{i_{1}}x_{2}^{i_{2}}\cdots x_{k}^{\left(  i_{k}+1\right)
-1}\cdots x_{N}^{i_{N}}}_{=x_{1}^{i_{1}}x_{2}^{i_{2}}\cdots x_{N}^{i_{N}}}\\
&  \ \ \ \ \ \ \ \ \ \ \ \ \ \ \ \ \ \ \ \ \left(  \text{here, we substituted
}i_{k}+1\text{ for }i_{k}\text{ in the sum}\right) \\
&  =\sum_{\substack{\left(  i_{1},i_{2},\ldots,i_{N}\right)  \in\mathbb{N}%
^{N};\\i_{1}+i_{2}+\cdots+i_{N}=n-1}}\left(  i_{k}+1\right)  x_{1}^{i_{1}%
}x_{2}^{i_{2}}\cdots x_{N}^{i_{N}}.
\end{align*}

Summing this equality over all $k\in\left[  N\right]  $, we obtain%
\begin{align*}
\sum_{k\in\left[  N\right]  }\dfrac{\partial}{\partial x_{k}}h_{n}  &
=\sum_{k\in\left[  N\right]  }\ \ \sum_{\substack{\left(  i_{1},i_{2}%
,\ldots,i_{N}\right)  \in\mathbb{N}^{N};\\i_{1}+i_{2}+\cdots+i_{N}%
=n-1}}\left(  i_{k}+1\right)  x_{1}^{i_{1}}x_{2}^{i_{2}}\cdots x_{N}^{i_{N}}\\
&  =\sum_{\substack{\left(  i_{1},i_{2},\ldots,i_{N}\right)  \in\mathbb{N}%
^{N};\\i_{1}+i_{2}+\cdots+i_{N}=n-1}}\underbrace{\left(  \sum_{k\in\left[
N\right]  }\left(  i_{k}+1\right)  \right)  }_{\substack{=\left(
i_{1}+1\right)  +\left(  i_{2}+1\right)  +\cdots+\left(  i_{N}+1\right)
\\=\left(  i_{1}+i_{2}+\cdots+i_{N}\right)  +N\\=n-1+N\\\text{(since }%
i_{1}+i_{2}+\cdots+i_{N}=n-1\text{)}}}x_{1}^{i_{1}}x_{2}^{i_{2}}\cdots
x_{N}^{i_{N}}\\
&  =\underbrace{\left(  n-1+N\right)  }_{=n+N-1}\underbrace{\sum
_{\substack{\left(  i_{1},i_{2},\ldots,i_{N}\right)  \in\mathbb{N}^{N}%
;\\i_{1}+i_{2}+\cdots+i_{N}=n-1}}x_{1}^{i_{1}}x_{2}^{i_{2}}\cdots x_{N}%
^{i_{N}}}_{\substack{=h_{n-1}\\\text{(by the definition of }h_{n-1}\text{)}%
}}=\left(  n+N-1\right)  h_{n-1}.
\end{align*}
This can be rewritten as $\nabla\left(  h_{n}\right)  =\left(  n+N-1\right)
h_{n-1}$ (since $\nabla=\dfrac{\partial}{\partial x_{1}}+\dfrac{\partial
}{\partial x_{2}}+\cdots+\dfrac{\partial}{\partial x_{N}}=\sum_{k\in\left[
N\right]  }\dfrac{\partial}{\partial x_{k}}$). Thus, Lemma \ref{lem.Nabla-h}
is proved.
\end{proof}

\section{Lemmas on determinants}

We will next need a few simple lemmas about determinants:

\begin{lemma}
\label{lem.det.s}Let $\lambda,\mu\in\mathcal{P}_{N}$. Define $\ell
,m\in\mathbb{Z}^{N}$ as in Theorem \ref{thm.1}. Then,%
\[
\det\left(  h_{\ell_{i}-m_{j}}\right)  _{i,j\in\left[  N\right]  }%
=s_{\lambda/\mu}.
\]

\end{lemma}

\begin{proof}
[Proof of Lemma \ref{lem.det.s}.]For every $i,j\in\left[  N\right]  $, we have
$\ell_{i}=\lambda_{i}-i$ (by the definition of $\ell$) and $m_{j}=\mu_{j}-j$
(similarly). Subtracting these two equalities from each other, we obtain that%
\[
\ell_{i}-m_{j}=\lambda_{i}-i-\left(  \mu_{j}-j\right)  =\lambda_{i}-\mu
_{j}-i+j\ \ \ \ \ \ \ \ \ \ \text{for every }i,j\in\left[  N\right]  .
\]
Hence, we can rewrite (\ref{pf.thm.1.JT}) as $s_{\lambda/\mu}=\det\left(
h_{\ell_{i}-m_{j}}\right)  _{i,j\in\left[  N\right]  }$. This proves Lemma
\ref{lem.det.s}.
\end{proof}

\begin{lemma}
\label{lem.det1.rows} Let $\lambda,\mu\in\mathcal{P}_{N}$ and $k\in\left[
N\right]  $ be such that $\lambda-e_{k}\in\mathcal{P}_{N}$. Define $\ell
,m\in\mathbb{Z}^{N}$ as in Theorem \ref{thm.1}. Then,
\[
\det\left(  h_{\left(  \ell-e_{k}\right)  _{i}-m_{j}}\right)  _{i,j\in\left[
N\right]  }=s_{\left(  \lambda-e_{k}\right)  /\mu}.
\]

\end{lemma}

\begin{proof}
[Proof of Lemma \ref{lem.det1.rows}.]Recall that the $N$-tuple $\ell$ is
defined from the $N$-tuple $\lambda$ by subtracting $1$ from its $1$-st entry,
subtracting $2$ from its $2$-nd entry, subtracting $3$ from its $3$-rd entry,
etc.. Thus, the $N$-tuple $\ell-e_{k}$ is obtained from $\lambda-e_{k}$ in the
same way. Hence, Lemma \ref{lem.det.s} (applied to $\lambda-e_{k}$ and
$\ell-e_{k}$ instead of $\lambda$ and $\ell$) yields%
\[
\det\left(  h_{\left(  \ell-e_{k}\right)  _{i}-m_{j}}\right)  _{i,j\in\left[
N\right]  }=s_{\left(  \lambda-e_{k}\right)  /\mu}.
\]
This proves Lemma \ref{lem.det1.rows}.
\end{proof}

\begin{lemma}
\label{lem.det1.cols} Let $\lambda,\mu\in\mathcal{P}_{N}$ and $k\in\left[
N\right]  $ be such that $\mu+e_{k}\in\mathcal{P}_{N}$. Define $\ell
,m\in\mathbb{Z}^{N}$ as in Theorem \ref{thm.1}. Then,
\[
\det\left(  h_{\ell_{i}-\left(  m+e_{k}\right)  _{j}}\right)  _{i,j\in\left[
N\right]  }=s_{\lambda/\left(  \mu+e_{k}\right)  }.
\]

\end{lemma}

\begin{proof}
[Proof of Lemma \ref{lem.det1.cols}.]Similarly to Lemma \ref{lem.det1.rows},
we can show this by applying Lemma \ref{lem.det.s} to $\mu+e_{k}$ and
$m+e_{k}$ instead of $\mu$ and $m$.
\end{proof}

\begin{lemma}
\label{lem.det0.rows} Let $\lambda,\mu\in\mathcal{P}_{N}$ and $k\in\left[
N\right]  $ be such that $\lambda-e_{k}\notin\mathcal{P}_{N}$. Define
$\ell,m\in\mathbb{Z}^{N}$ as in Theorem \ref{thm.1}. Then,
\[
\det\left(  h_{\left(  \ell-e_{k}\right)  _{i}-m_{j}}\right)  _{i,j\in\left[
N\right]  }=0.
\]

\end{lemma}

\begin{proof}
[Proof of Lemma \ref{lem.det0.rows}.]The definition of $e_{k}$ yields that the
only nonzero entry of $e_{k}$ is $\left(  e_{k}\right)  _{k}=1$. Hence, in
particular, $\left(  e_{k}\right)  _{k+1}=0$.

We have $\lambda\in\mathcal{P}_{N}$, so that $\lambda_{1}\geq\lambda_{2}%
\geq\cdots\geq\lambda_{N}\geq0$. Recall that $k\in\left[  N\right]  $, so that
$k\leq N$. Hence, we are in one of the following two cases:

\textit{Case 1:} We have $k<N$.

\textit{Case 2:} We have $k=N$.

Let us first consider Case 1. In this case, we have $k<N$. The definition of
$e_{k}$ yields $\lambda-e_{k}=\left(  \lambda_{1},\lambda_{2},\ldots
,\lambda_{k-1},\lambda_{k}-1,\lambda_{k+1},\ldots,\lambda_{N}\right)  $ (this
is the $N$-tuple $\lambda$ with its $k$-th entry decreased by $1$). Thus, the
chain of inequalities
\[
\lambda_{1}\geq\lambda_{2}\geq\cdots\geq\lambda_{k-1}\geq\lambda_{k}%
-1\geq\lambda_{k+1}\geq\cdots\geq\lambda_{N}\geq0\quad\text{\textbf{does not}
hold}%
\]
(since $\lambda-e_{k}\notin\mathcal{P}_{N}$). Therefore, the inequality
$\lambda_{k}-1\geq\lambda_{k+1}$ must be violated (since all the other
inequality signs in this chain follow from $\lambda_{1}\geq\lambda_{2}%
\geq\cdots\geq\lambda_{N}\geq0$). In other words, we have $\lambda
_{k}-1<\lambda_{k+1}$. Since $\lambda_{k}$ and $\lambda_{k+1}$ are integers,
this entails $\lambda_{k}-1\leq\lambda_{k+1}-1$, so that $\lambda_{k}%
\leq\lambda_{k+1}$. Combining this with $\lambda_{k}\geq\lambda_{k+1}$ (which
follows from $\lambda_{1}\geq\lambda_{2}\geq\cdots\geq\lambda_{N}$), we obtain
$\lambda_{k}=\lambda_{k+1}$.

Now, the definition of $\ell$ yields $\ell_{k}=\lambda_{k}-k$. Hence,
\[
\left(  \ell-e_{k}\right)  _{k}=\underbrace{\ell_{k}}_{=\lambda_{k}%
-k}-\underbrace{\left(  e_{k}\right)  _{k}}_{=1}=\underbrace{\lambda_{k}%
}_{=\lambda_{k+1}}-\,k-1=\lambda_{k+1}-k-1.
\]
Furthermore, the definition of $\ell$ yields $\ell_{k+1}=\lambda_{k+1}-\left(
k+1\right)  $. Hence,%
\[
\left(  \ell-e_{k}\right)  _{k+1}=\ell_{k+1}-\underbrace{\left(  e_{k}\right)
_{k+1}}_{=0}=\ell_{k+1}=\lambda_{k+1}-\left(  k+1\right)  =\lambda_{k+1}-k-1.
\]
Comparing this with $\left(  \ell-e_{k}\right)  _{k}=\lambda_{k+1}-k-1$, we
obtain $\left(  \ell-e_{k}\right)  _{k}=\left(  \ell-e_{k}\right)  _{k+1}$.

Hence, for each $j\in\left[  N\right]  $, we have%
\[
h_{\left(  \ell-e_{k}\right)  _{k}-m_{j}}=h_{\left(  \ell-e_{k}\right)
_{k+1}-m_{j}}.
\]
In other words, each entry in the $k$-th row of the matrix $\left(  h_{\left(
\ell-e_{k}\right)  _{i}-m_{j}}\right)  _{i,j\in\left[  N\right]  }$ equals the
corresponding entry in the $\left(  k+1\right)  $-st row of this matrix. Thus,
the matrix $\left(  h_{\left(  \ell-e_{k}\right)  _{i}-m_{j}}\right)
_{i,j\in\left[  N\right]  }$ has two equal rows (namely, its $k$-th and
$\left(  k+1\right)  $-st rows). Hence, its determinant is $0$. This proves
Lemma \ref{lem.det0.rows} in Case 1.

Let us now consider Case 2. In this case, we have $k=N$. Hence, $\lambda
-e_{k}=\lambda-e_{N}=\left(  \lambda_{1},\lambda_{2},\ldots,\lambda
_{N-1},\lambda_{N}-1\right)  $ (this is the $N$-tuple $\lambda$ with its
$N$-th entry decreased by $1$). Thus, the chain of inequalities
\[
\lambda_{1}\geq\lambda_{2}\geq\cdots\geq\lambda_{N-1}\geq\lambda_{N}%
-1\geq0\quad\text{\textbf{does not} hold}%
\]
(since $\lambda-e_{k}\notin\mathcal{P}_{N}$). Therefore, the inequality
$\lambda_{N}-1\geq0$ must be violated (since all the other inequality signs in
this chain follow from $\lambda_{1}\geq\lambda_{2}\geq\cdots\geq\lambda
_{N}\geq0$). In other words, we have $\lambda_{N}-1<0$. In other words,
$\lambda_{k}-1<0$ (since $k=N$). Thus, $\lambda_{k}<1$. Now,
\[
\left(  \ell-e_{k}\right)  _{k}=\underbrace{\ell_{k}}_{\substack{=\lambda
_{k}-k\\\text{(by the definition of }\ell\text{)}}}-\underbrace{\left(
e_{k}\right)  _{k}}_{=1}=\underbrace{\lambda_{k}}_{<1}-\underbrace{k}%
_{=N}-\,1<1-N-1=-N
\]

On the other hand, for each $j\in\left[  N\right]  $, we have $m_{j}=\mu
_{j}-j$ (by the definition of $m$) and thus $m_{j}=\underbrace{\mu_{j}}%
_{\geq0}-\underbrace{j}_{\leq N}\geq0-N=-N>\left(  \ell-e_{k}\right)  _{k}$
(since $\left(  \ell-e_{k}\right)  _{k}<-N$). Thus, for each $j\in\left[
N\right]  $, we have $\left(  \ell-e_{k}\right)  _{k}-m_{j}<0$ and therefore%
\[
h_{\left(  \ell-e_{k}\right)  _{k}-m_{j}}=0\ \ \ \ \ \ \ \ \ \ \left(
\text{since }h_{i}=0\text{ for all }i<0\right)  .
\]
In other words, each entry in the $k$-th row of the matrix $\left(  h_{\left(
\ell-e_{k}\right)  _{i}-m_{j}}\right)  _{i,j\in\left[  N\right]  }$ is zero.
Thus, the matrix $\left(  h_{\left(  \ell-e_{k}\right)  _{i}-m_{j}}\right)
_{i,j\in\left[  N\right]  }$ has a zero row (namely, its $k$-th row). Hence,
its determinant is $0$. This proves Lemma \ref{lem.det0.rows} in Case 2.

We have now proved Lemma \ref{lem.det0.rows} in both cases.
\end{proof}

\begin{lemma}
\label{lem.det0.cols} Let $\lambda,\mu\in\mathcal{P}_{N}$ and $k\in\left[
N\right]  $ be such that $\mu+e_{k}\notin\mathcal{P}_{N}$. Define $\ell
,m\in\mathbb{Z}^{N}$ as in Theorem \ref{thm.1}. Then,
\[
\det\left(  h_{\ell_{i}-\left(  m+e_{k}\right)  _{j}}\right)  _{i,j\in\left[
N\right]  }=0.
\]

\end{lemma}

\begin{proof}
[Proof of Lemma \ref{lem.det0.cols}.]This is similar to Case 2 in the proof of
Lemma \ref{lem.det0.rows}. Here are the details:

The definition of $e_{k}$ yields $\mu+e_{k}=\left(  \mu_{1},\mu_{2},\ldots
,\mu_{k-1},\mu_{k}+1,\mu_{k+1},\ldots,\mu_{N}\right)  $ (this is the $N$-tuple
$\mu$ with its $k$-th entry increased by $1$).

We have $\mu\in\mathcal{P}_{N}$, so that $\mu_{1}\geq\mu_{2}\geq\cdots\geq
\mu_{N}\geq0$. However, we have $\mu+e_{k}\notin\mathcal{P}_{N}$, so that the
chain of inequalities
\[
\mu_{1}\geq\mu_{2}\geq\cdots\geq\mu_{k-1}\geq\mu_{k}+1\geq\mu_{k+1}\geq
\cdots\geq\mu_{N}\geq0\quad\text{\textbf{does not} hold}%
\]
(since $\mu+e_{k}=\left(  \mu_{1},\mu_{2},\ldots,\mu_{k-1},\mu_{k}+1,\mu
_{k+1},\ldots,\mu_{N}\right)  $). Hence, the inequality $\mu_{k-1}\geq\mu
_{k}+1$ must be violated (since all the other inequalities in this chain
follow from $\mu_{1}\geq\mu_{2}\geq\cdots\geq\mu_{N}\geq0$). In other words,
we must have $k>1$ and $\mu_{k-1}<\mu_{k}+1$.

From $\mu_{k-1}<\mu_{k}+1$, we obtain $\mu_{k-1}\leq\mu_{k}$ (since $\mu
_{k-1}$ and $\mu_{k}$ are integers). Combining this with $\mu_{k-1}\geq\mu
_{k}$ (since $\mu_{1}\geq\mu_{2}\geq\cdots\geq\mu_{N}$), we obtain $\mu
_{k-1}=\mu_{k}$.

The definition of $e_{k}$ yields that the only nonzero entry of $e_{k}$ is
$\left(  e_{k}\right)  _{k}=1$. Hence, in particular, $\left(  e_{k}\right)
_{k-1}=0$.

The definition of $m$ yields $m_{k}=\mu_{k}-k$ and $m_{k-1}=\mu_{k-1}-\left(
k-1\right)  $. Now,%
\[
\left(  m+e_{k}\right)  _{k}=\underbrace{m_{k}}_{=\mu_{k}-k}%
+\underbrace{\left(  e_{k}\right)  _{k}}_{=1}=\mu_{k}-k+1=\mu_{k}-\left(
k-1\right)  .
\]
Comparing this with%
\[
\left(  m+e_{k}\right)  _{k-1}=m_{k-1}+\underbrace{\left(  e_{k}\right)
_{k-1}}_{=0}=m_{k-1}=\underbrace{\mu_{k-1}}_{=\mu_{k}}-\left(  k-1\right)
=\mu_{k}-\left(  k-1\right)  ,
\]
we obtain $\left(  m+e_{k}\right)  _{k}=\left(  m+e_{k}\right)  _{k-1}$.

Now, for each $i\in\left[  N\right]  $, we have%
\[
h_{\ell_{i}-\left(  m+e_{k}\right)  _{k}}=h_{\ell_{i}-\left(  m+e_{k}\right)
_{k-1}}\ \ \ \ \ \ \ \ \ \ \left(  \text{since }\left(  m+e_{k}\right)
_{k}=\left(  m+e_{k}\right)  _{k-1}\right)  .
\]
In other words, each entry in the $k$-th column of the matrix $\left(
h_{\ell_{i}-\left(  m+e_{k}\right)  _{j}}\right)  _{i,j\in\left[  N\right]  }$
equals the corresponding entry in the $\left(  k-1\right)  $-st column of this
matrix. Thus, the matrix $\left(  h_{\ell_{i}-\left(  m+e_{k}\right)  _{j}%
}\right)  _{i,j\in\left[  N\right]  }$ has two equal columns (namely, its
$k$-th and $\left(  k-1\right)  $-st columns). Hence, its determinant is $0$.
This proves Lemma \ref{lem.det0.cols}.
\end{proof}

Finally, we will need the Leibniz rule for products of multiple
factors:\footnote{Recall that $R=\mathbb{Z}\left[  x_{1},x_{2},\ldots
,x_{n}\right]  $.}

\begin{lemma}
\label{lem.leib-n}For any $a_{1},a_{2},\ldots,a_{n}\in R$, we have
\[
\nabla\left(  a_{1}a_{2}\cdots a_{n}\right)  =\sum_{k=1}^{n}a_{1}a_{2}\cdots
a_{k-1}\nabla\left(  a_{k}\right)  a_{k+1}a_{k+2}\cdots a_{n}.
\]

\end{lemma}

\begin{proof}
[Proof of Lemma \ref{lem.leib-n}.]This holds not just for $\nabla$ but
actually for any derivation of any ring, and can be easily proved by induction
on $n$ using the Leibniz rule.
\end{proof}

\section{The last lemma}

We now have everything in place for the proofs of Theorem \ref{thm.1} and
Corollary \ref{cor.2}. However, to keep our computation short, let us
outsource a part of it to a lemma:

\begin{lemma}
\label{lem.Dprod}Let $\sigma$ be a permutation of the set $\left[  N\right]
$. Let $\ell,m\in\mathbb{Z}^{N}$ be two $N$-tuples of integers. Let $a$ and
$b$ be two integers such that $a+b=N-1$. Then,%
\[
\nabla\left(  \prod_{i=1}^{N}h_{\ell_{i}-m_{\sigma\left(  i\right)  }}\right)
=\sum_{k=1}^{N}\left(  \ell_{k}+a\right)  \prod_{i=1}^{N}h_{\left(  \ell
-e_{k}\right)  _{i}-m_{\sigma\left(  i\right)  }}+\sum_{k=1}^{N}\left(
b-m_{k}\right)  \prod_{i=1}^{N}h_{\ell_{i}-\left(  m+e_{k}\right)
_{\sigma\left(  i\right)  }}.
\]

\end{lemma}

\begin{proof}
[Proof of Lemma \ref{lem.Dprod}.]Fix $k\in\left[  N\right]  $. The $N$-tuple
$\ell-e_{k}$ differs from the $N$-tuple $\ell$ only in its $k$-th entry, which
is smaller by $1$. Thus, the product $\prod_{i=1}^{N}h_{\left(  \ell
-e_{k}\right)  _{i}-m_{\sigma\left(  i\right)  }}$ differs from the product
$\prod_{i=1}^{N}h_{\ell_{i}-m_{\sigma\left(  i\right)  }}$ only in its $k$-th
factor, which is $h_{\left(  \ell_{k}-1\right)  -m_{\sigma\left(  k\right)  }%
}$ instead of $h_{\ell_{k}-m_{\sigma\left(  k\right)  }}$. Hence,%
\begin{align}
\prod_{i=1}^{N}h_{\left(  \ell-e_{k}\right)  _{i}-m_{\sigma\left(  i\right)
}}  &  =\underbrace{h_{\left(  \ell_{k}-1\right)  -m_{\sigma\left(  k\right)
}}}_{=h_{\ell_{k}-m_{\sigma\left(  k\right)  }-1}}\ \ \prod_{\substack{i\in
\left[  N\right]  ;\\i\neq k}}h_{\ell_{i}-m_{\sigma\left(  i\right)  }%
}\nonumber\\
&  =h_{\ell_{k}-m_{\sigma\left(  k\right)  }-1}\prod_{\substack{i\in\left[
N\right]  ;\\i\neq k}}h_{\ell_{i}-m_{\sigma\left(  i\right)  }}.
\label{pf.lem.Dprod.pr1}%
\end{align}
The $N$-tuple $m+e_{\sigma\left(  k\right)  }$ differs from the $N$-tuple $m$
only in its $\sigma\left(  k\right)  $-th entry, which is larger by $1$. Thus,
the product $\prod_{i=1}^{N}h_{\ell_{i}-\left(  m+e_{\sigma\left(  k\right)
}\right)  _{\sigma\left(  i\right)  }}$ differs from the product $\prod
_{i=1}^{N}h_{\ell_{i}-m_{\sigma\left(  i\right)  }}$ only in its $k$-th
factor, which is $h_{\ell_{k}-\left(  m_{\sigma\left(  k\right)  }+1\right)
}$ instead of $h_{\ell_{k}-m_{\sigma\left(  k\right)  }}$ (since all the
remaining factors satisfy $i\neq k$ and therefore $\sigma\left(  i\right)
\neq\sigma\left(  k\right)  $). Hence,%
\begin{align}
\prod_{i=1}^{N}h_{\ell_{i}-\left(  m+e_{\sigma\left(  k\right)  }\right)
_{\sigma\left(  i\right)  }}  &  =\underbrace{h_{\ell_{k}-\left(
m_{\sigma\left(  k\right)  }+1\right)  }}_{=h_{\ell_{k}-m_{\sigma\left(
k\right)  }-1}}\ \ \prod_{\substack{i\in\left[  N\right]  ;\\i\neq k}%
}h_{\ell_{i}-m_{\sigma\left(  i\right)  }}\nonumber\\
&  =h_{\ell_{k}-m_{\sigma\left(  k\right)  }-1}\prod_{\substack{i\in\left[
N\right]  ;\\i\neq k}}h_{\ell_{i}-m_{\sigma\left(  i\right)  }}.
\label{pf.lem.Dprod.pr2}%
\end{align}

Forget that we fixed $k$. We thus have proved the equalities
(\ref{pf.lem.Dprod.pr1}) and (\ref{pf.lem.Dprod.pr2}) for each $k\in\left[
N\right]  $.

Since the ring $R$ is commutative, we can rewrite Lemma \ref{lem.leib-n} as
follows: For any $a_{1},a_{2},\ldots,a_{n}\in R$, we have
\[
\nabla\left(  \prod_{i=1}^{n}a_{i}\right)  =\sum_{k=1}^{n}\nabla\left(
a_{k}\right)  \prod_{\substack{i\in\left[  n\right]  ;\\i\neq k}}a_{i}.
\]

Applying this to $N=n$ and $a_{i}=h_{\ell_{i}-m_{\sigma\left(  i\right)  }}$,
we obtain%
\begin{align*}
&  \nabla\left(  \prod_{i=1}^{N}h_{\ell_{i}-m_{\sigma\left(  i\right)  }%
}\right)  \\
&  =\sum_{k=1}^{N}\underbrace{\nabla\left(  h_{\ell_{k}-m_{\sigma\left(
k\right)  }}\right)  }_{\substack{=\left(  \ell_{k}-m_{\sigma\left(  k\right)
}+N-1\right)  h_{\ell_{k}-m_{\sigma\left(  k\right)  }-1}\\\text{(by Lemma
\ref{lem.Nabla-h})}}}\prod_{\substack{i\in\left[  N\right]  ;\\i\neq
k}}h_{\ell_{i}-m_{\sigma\left(  i\right)  }}\\
&  =\sum_{k=1}^{N}\underbrace{\left(  \ell_{k}-m_{\sigma\left(  k\right)
}+N-1\right)  }_{\substack{=\left(  \ell_{k}+a\right)  +\left(  b-m_{\sigma
\left(  k\right)  }\right)  \\\text{(since }N-1=a+b\text{)}}}h_{\ell
_{k}-m_{\sigma\left(  k\right)  }-1}\prod_{\substack{i\in\left[  N\right]
;\\i\neq k}}h_{\ell_{i}-m_{\sigma\left(  i\right)  }}\\
&  =\sum_{k=1}^{N}\left(  \left(  \ell_{k}+a\right)  +\left(  b-m_{\sigma
\left(  k\right)  }\right)  \right)  h_{\ell_{k}-m_{\sigma\left(  k\right)
}-1}\prod_{\substack{i\in\left[  N\right]  ;\\i\neq k}}h_{\ell_{i}%
-m_{\sigma\left(  i\right)  }}\\
&  =\sum_{k=1}^{N}\left(  \ell_{k}+a\right)  \underbrace{h_{\ell_{k}%
-m_{\sigma\left(  k\right)  }-1}\prod_{\substack{i\in\left[  N\right]
;\\i\neq k}}h_{\ell_{i}-m_{\sigma\left(  i\right)  }}}_{\substack{=\prod
_{i=1}^{N}h_{\left(  \ell-e_{k}\right)  _{i}-m_{\sigma\left(  i\right)  }%
}\\\text{(by (\ref{pf.lem.Dprod.pr1}))}}}+\sum_{k=1}^{N}\left(  b-m_{\sigma
\left(  k\right)  }\right)  \underbrace{h_{\ell_{k}-m_{\sigma\left(  k\right)
}-1}\prod_{\substack{i\in\left[  N\right]  ;\\i\neq k}}h_{\ell_{i}%
-m_{\sigma\left(  i\right)  }}}_{\substack{=\prod_{i=1}^{N}h_{\ell_{i}-\left(
m+e_{\sigma\left(  k\right)  }\right)  _{\sigma\left(  i\right)  }}\\\text{(by
(\ref{pf.lem.Dprod.pr2}))}}}\\
&  =\sum_{k=1}^{N}\left(  \ell_{k}+a\right)  \prod_{i=1}^{N}h_{\left(
\ell-e_{k}\right)  _{i}-m_{\sigma\left(  i\right)  }}+\sum_{k=1}^{N}\left(
b-m_{\sigma\left(  k\right)  }\right)  \prod_{i=1}^{N}h_{\ell_{i}-\left(
m+e_{\sigma\left(  k\right)  }\right)  _{\sigma\left(  i\right)  }}\\
&  =\sum_{k=1}^{N}\left(  \ell_{k}+a\right)  \prod_{i=1}^{N}h_{\left(
\ell-e_{k}\right)  _{i}-m_{\sigma\left(  i\right)  }}+\sum_{k=1}^{N}\left(
b-m_{k}\right)  \prod_{i=1}^{N}h_{\ell_{i}-\left(  m+e_{k}\right)
_{\sigma\left(  i\right)  }}%
\end{align*}
(here, we have substituted $k$ for $\sigma\left(  k\right)  $ in the second
sum, since $\sigma:\left[  N\right]  \rightarrow\left[  N\right]  $ is a
bijection). This proves Lemma \ref{lem.Dprod}.
\end{proof}

\section{Proofs of the main results}

\begin{proof}
[Proof of Theorem \ref{thm.1}.]Let $S_{N}$ denote the $N$-th symmetric group
(i.e., the group of all permutations of $\left[  N\right]  $). Let $\left(
-1\right)  ^{\sigma}$ denote the sign of any permutation $\sigma$. The
definition of a determinant says that
\begin{equation}
\det\left(  a_{i,j}\right)  _{i,j\in\left[  N\right]  }=\sum_{\sigma\in S_{N}%
}\left(  -1\right)  ^{\sigma}\prod_{i=1}^{N}a_{i,\sigma\left(  i\right)  }
\label{pf.thm.1.det=}%
\end{equation}
for any matrix $\left(  a_{i,j}\right)  _{i,j\in\left[  N\right]  }\in
R^{N\times N}$. Now, Lemma \ref{lem.det.s} yields%
\[
s_{\lambda/\mu}=\det\left(  h_{\ell_{i}-m_{j}}\right)  _{i,j\in\left[
N\right]  }=\sum_{\sigma\in S_{N}}\left(  -1\right)  ^{\sigma}\prod_{i=1}%
^{N}h_{\ell_{i}-m_{\sigma\left(  i\right)  }}%
\]
(by (\ref{pf.thm.1.det=})). Applying the $\mathbb{Z}$-linear map $\nabla$ to
both sides of this equality, we find%
\begin{align*}
&  \nabla\left(  s_{\lambda/\mu}\right) \\
&  =\sum_{\sigma\in S_{N}}\left(  -1\right)  ^{\sigma}\nabla\left(
\prod_{i=1}^{N}h_{\ell_{i}-m_{\sigma\left(  i\right)  }}\right) \\
&  =\sum_{\sigma\in S_{N}}\left(  -1\right)  ^{\sigma}\left(  \sum_{k=1}%
^{N}\left(  \ell_{k}+a\right)  \prod_{i=1}^{N}h_{\left(  \ell-e_{k}\right)
_{i}-m_{\sigma\left(  i\right)  }}+\sum_{k=1}^{N}\left(  b-m_{k}\right)
\prod_{i=1}^{N}h_{\ell_{i}-\left(  m+e_{k}\right)  _{\sigma\left(  i\right)
}}\right) \\
&  \qquad\qquad\left(  \text{by Lemma \ref{lem.Dprod}}\right) \\
&  =\sum_{\sigma\in S_{N}}\left(  -1\right)  ^{\sigma}\sum_{k=1}^{N}\left(
\ell_{k}+a\right)  \prod_{i=1}^{N}h_{\left(  \ell-e_{k}\right)  _{i}%
-m_{\sigma\left(  i\right)  }}+\sum_{\sigma\in S_{N}}\left(  -1\right)
^{\sigma}\sum_{k=1}^{N}\left(  b-m_{k}\right)  \prod_{i=1}^{N}h_{\ell
_{i}-\left(  m+e_{k}\right)  _{\sigma\left(  i\right)  }}.
\end{align*}
In view of%
\begin{align*}
&  \sum_{\sigma\in S_{N}}\left(  -1\right)  ^{\sigma}\sum_{k=1}^{N}\left(
\ell_{k}+a\right)  \prod_{i=1}^{N}h_{\left(  \ell-e_{k}\right)  _{i}%
-m_{\sigma\left(  i\right)  }}\\
&  =\sum_{k=1}^{N}\left(  \ell_{k}+a\right)  \underbrace{\sum_{\sigma\in
S_{N}}\left(  -1\right)  ^{\sigma}\prod_{i=1}^{N}h_{\left(  \ell-e_{k}\right)
_{i}-m_{\sigma\left(  i\right)  }}}_{\substack{=\det\left(  h_{\left(
\ell-e_{k}\right)  _{i}-m_{j}}\right)  _{i,j\in\left[  N\right]  }\\\text{(by
(\ref{pf.thm.1.det=}))}}}\\
&  =\sum_{k=1}^{N}\left(  \ell_{k}+a\right)  \underbrace{\det\left(
h_{\left(  \ell-e_{k}\right)  _{i}-m_{j}}\right)  _{i,j\in\left[  N\right]  }%
}_{\substack{=0\text{ if }\lambda-e_{k}\notin\mathcal{P}_{N}\\\text{(by Lemma
\ref{lem.det0.rows})}}}\\
&  =\sum_{\substack{k\in\left[  N\right]  ;\\\lambda-e_{k}\in\mathcal{P}_{N}%
}}\left(  \ell_{k}+a\right)  \underbrace{\det\left(  h_{\left(  \ell
-e_{k}\right)  _{i}-m_{j}}\right)  _{i,j\in\left[  N\right]  }}%
_{\substack{=s_{\left(  \lambda-e_{k}\right)  /\mu}\\\text{(by Lemma
\ref{lem.det1.rows})}}}\\
&  =\sum_{\substack{k\in\left[  N\right]  ;\\\lambda-e_{k}\in\mathcal{P}_{N}%
}}\left(  \ell_{k}+a\right)  s_{\left(  \lambda-e_{k}\right)  /\mu}%
=\sum_{\substack{i\in\left[  N\right]  ;\\\lambda-e_{i}\in\mathcal{P}_{N}%
}}\left(  \ell_{i}+a\right)  s_{\left(  \lambda-e_{i}\right)  /\mu}%
\end{align*}
and%
\begin{align*}
&  \sum_{\sigma\in S_{N}}\left(  -1\right)  ^{\sigma}\sum_{k=1}^{N}\left(
b-m_{k}\right)  \prod_{i=1}^{N}h_{\ell_{i}-\left(  m+e_{k}\right)
_{\sigma\left(  i\right)  }}\\
&  =\sum_{k=1}^{N}\left(  b-m_{k}\right)  \underbrace{\sum_{\sigma\in S_{N}%
}\left(  -1\right)  ^{\sigma}\prod_{i=1}^{N}h_{\ell_{i}-\left(  m+e_{k}%
\right)  _{\sigma\left(  i\right)  }}}_{\substack{=\det\left(  h_{\ell
_{i}-\left(  m+e_{k}\right)  _{j}}\right)  _{i,j\in\left[  N\right]
}\\\text{(by (\ref{pf.thm.1.det=}))}}}\\
&  =\sum_{k=1}^{N}\left(  b-m_{k}\right)  \underbrace{\det\left(  h_{\ell
_{i}-\left(  m+e_{k}\right)  _{j}}\right)  _{i,j\in\left[  N\right]  }%
}_{\substack{=0\text{ if }\mu+e_{k}\notin\mathcal{P}_{N}\\\text{(by Lemma
\ref{lem.det0.cols})}}}\\
&  =\sum_{\substack{k\in\left[  N\right]  ;\\\mu+e_{k}\in\mathcal{P}_{N}%
}}\left(  b-m_{k}\right)  \underbrace{\det\left(  h_{\ell_{i}-\left(
m+e_{k}\right)  _{j}}\right)  _{i,j\in\left[  N\right]  }}%
_{\substack{=s_{\lambda/\left(  \mu+e_{k}\right)  }\\\text{(by Lemma
\ref{lem.det1.cols})}}}\\
&  =\sum_{\substack{k\in\left[  N\right]  ;\\\mu+e_{k}\in\mathcal{P}_{N}%
}}\left(  b-m_{k}\right)  s_{\lambda/\left(  \mu+e_{k}\right)  }%
=\sum_{\substack{i\in\left[  N\right]  ;\\\mu+e_{i}\in\mathcal{P}_{N}}}\left(
b-m_{i}\right)  s_{\lambda/\left(  \mu+e_{i}\right)  },
\end{align*}
this can be rewritten as%
\[
\nabla\left(  s_{\lambda/\mu}\right)  =\sum_{\substack{i\in\left[  N\right]
;\\\lambda-e_{i}\in\mathcal{P}_{N}}}\left(  \ell_{i}+a\right)  s_{\left(
\lambda-e_{i}\right)  /\mu}+\sum_{\substack{i\in\left[  N\right]  ;\\\mu
+e_{i}\in\mathcal{P}_{N}}}\left(  b-m_{i}\right)  s_{\lambda/\left(  \mu
+e_{i}\right)  }.
\]
Thus, Theorem \ref{thm.1} is proved.
\end{proof}

\begin{proof}
[Proof of Corollary \ref{cor.2}.]Define $\ell,m\in\mathbb{Z}^{N}$ as in
Theorem \ref{thm.1}. Theorem \ref{thm.1} (applied to $a=N-1$ and $b=0$) yields%
\[
\nabla\left(  s_{\lambda/\mu}\right)  =\sum_{\substack{i\in\left[  N\right]
;\\\lambda-e_{i}\in\mathcal{P}_{N}}}\left(  \ell_{i}+N-1\right)  s_{\left(
\lambda-e_{i}\right)  /\mu}+\sum_{\substack{i\in\left[  N\right]  ;\\\mu
+e_{i}\in\mathcal{P}_{N}}}\left(  0-m_{i}\right)  s_{\lambda/\left(  \mu
+e_{i}\right)  }.
\]
Theorem \ref{thm.1} (applied to $a=N$ and $b=-1$) yields%
\[
\nabla\left(  s_{\lambda/\mu}\right)  =\sum_{\substack{i\in\left[  N\right]
;\\\lambda-e_{i}\in\mathcal{P}_{N}}}\left(  \ell_{i}+N\right)  s_{\left(
\lambda-e_{i}\right)  /\mu}+\sum_{\substack{i\in\left[  N\right]  ;\\\mu
+e_{i}\in\mathcal{P}_{N}}}\left(  -1-m_{i}\right)  s_{\lambda/\left(
\mu+e_{i}\right)  }.
\]
Subtracting the latter equality from the former, we find%
\begin{align*}
0  &  =\sum_{\substack{i\in\left[  N\right]  ;\\\lambda-e_{i}\in
\mathcal{P}_{N}}}\left(  -1\right)  s_{\left(  \lambda-e_{i}\right)  /\mu
}+\sum_{\substack{i\in\left[  N\right]  ;\\\mu+e_{i}\in\mathcal{P}_{N}%
}}s_{\lambda/\left(  \mu+e_{i}\right)  }\\
&  =-\sum_{\substack{i\in\left[  N\right]  ;\\\lambda-e_{i}\in\mathcal{P}_{N}%
}}s_{\left(  \lambda-e_{i}\right)  /\mu}+\sum_{\substack{i\in\left[  N\right]
;\\\mu+e_{i}\in\mathcal{P}_{N}}}s_{\lambda/\left(  \mu+e_{i}\right)  }.
\end{align*}
In other words,
\[
\sum_{\substack{i\in\left[  N\right]  ;\\\lambda-e_{i}\in\mathcal{P}_{N}%
}}s_{\left(  \lambda-e_{i}\right)  /\mu}=\sum_{\substack{i\in\left[  N\right]
;\\\mu+e_{i}\in\mathcal{P}_{N}}}s_{\lambda/\left(  \mu+e_{i}\right)  }.
\]
This proves Corollary \ref{cor.2}.
\end{proof}

\section{Final remarks}

\textbf{1.} Clearly, Theorem \ref{thm.1} can be generalized by replacing
$\mathbb{Z}$ with any commutative ring $\mathbf{k}$. In this generality, $a$
and $b$ can be any two elements of $\mathbf{k}$ (rather than just integers)
satisfying $a+b=\left(  N-1\right)  \cdot1_{\mathbf{k}}$. However, not much
generality is gained in this way, since Corollary \ref{cor.2} easily shows
that all choices of $a$ and $b$ lead to the same sum.

\medskip

\textbf{2.} Theorem \ref{thm.1} can also be lifted to the \textquotedblleft
infinite setting\textquotedblright, i.e., to the ring of symmetric functions
in infinitely many variables (see, e.g., \cite[Chapter 7]{EC2} or
\cite[Chapter 2]{GriRei} for introductions to this ring). This is not
completely straightforward, since the diagonal derivative $\nabla$ is defined
only for finitely many indeterminates and depends on their number $N$ (for
instance, $\nabla\left(  s_{\left(  1\right)  }\right)  =N$). In the infinite
setting, it has to be replaced by a derivation $\nabla_{q}$ depending on a
scalar $q$:

Let $\mathbf{k}$ be a commutative ring, and let $q\in\mathbf{k}$ be an
element. (For example, we can have $\mathbf{k}=\mathbb{Z}\left[  q\right]  $
and $q=q$.) Let $\Lambda$ be the ring of symmetric functions in infinitely
many indeterminates $x_{1},x_{2},x_{3},\ldots$ over $\mathbf{k}$. For each
$n\in\mathbb{Z}$, we let $h_{n}$ denote the $n$-th complete homogeneous
symmetric function in $\Lambda$ (so that $h_{0}=1$ and $h_{i}=0$ for all
$i<0$). Let $\nabla_{q}:\Lambda\rightarrow\Lambda$ be the unique derivation
that satisfies%
\[
\nabla_{q}\left(  h_{n}\right)  =\left(  n+q-1\right)  h_{n-1}%
\ \ \ \ \ \ \ \ \ \ \text{for each }n>0.
\]
\footnote{This condition uniquely determines a derivation of $\Lambda$, since
the elements $h_{1},h_{2},h_{3},\ldots$ freely generate $\Lambda$ as a
commutative $\mathbf{k}$-algebra. It is easy to see that this derivation
$\nabla_{q}$ satisfies the equality $\nabla_{q}\left(  h_{n}\right)  =\left(
n+q-1\right)  h_{n-1}$ for all $n\in\mathbb{Z}$.} For $q=N\in\mathbb{N}$, this
derivation is a lift of the directional derivative operator $\nabla
:R\rightarrow R$ to $\Lambda$ (meaning that $\nabla\circ\pi=\pi\circ\nabla
_{N}$, where $\pi:\Lambda\rightarrow R$ is the evaluation homomorphism at
$x_{1},x_{2},\ldots,x_{N},0,0,0,\ldots$). With these definitions, we can
extend Theorem \ref{thm.1} to a general property of $\nabla_{q}$:

\begin{theorem}
\label{thm.3} Let $a$ and $b$ be two elements of $\mathbf{k}$ such that
$a+b=q-1$. Let $\lambda=\left(  \lambda_{1},\lambda_{2},\lambda_{3}%
,\ldots\right)  $ and $\mu=\left(  \mu_{1},\mu_{2},\mu_{3},\ldots,\right)  $
be two partitions. Set%
\[
\ell_{i}:=\lambda_{i}-i\ \ \ \ \ \ \ \ \ \ \text{and}\ \ \ \ \ \ \ \ \ \ m_{i}%
:=\mu_{i}-i\ \ \ \ \ \ \ \ \ \ \text{for each }i\geq1.
\]
Then,
\[
\nabla_{q}\left(  s_{\lambda/\mu}\right)  =\sum_{\substack{i\geq
1;\\\lambda-e_{i}\text{ is a partition}}}\left(  \ell_{i}+a\right)  s_{\left(
\lambda-e_{i}\right)  /\mu}+\sum_{\substack{i\geq1;\\\mu+e_{i}\text{ is a
partition}}}\left(  b-m_{i}\right)  s_{\lambda/\left(  \mu+e_{i}\right)  }.
\]
Here, $e_{i}$ means the infinite sequence $\left(  0,0,\ldots,0,1,0,0,0,\ldots
\right)  $ with the $1$ in its $i$-th position.
\end{theorem}

The proof of this theorem is similar to our above proof of Theorem \ref{thm.1}
(with the minor complication that we have to fix an $N\in\mathbb{N}$ that is
strictly larger than the lengths of $\lambda$ and $\mu$, in order to apply the
Jacobi--Trudi formula), and is left to the reader.

\medskip

\textbf{3.} It is natural to attempt generalizing Theorem \ref{thm.1} to
higher-order differential operators, such as $\nabla^{\prime}:=\dfrac
{\partial^{2}}{\partial x_{1}^{2}}+\dfrac{\partial^{2}}{\partial x_{2}^{2}%
}+\cdots+\dfrac{\partial^{2}}{\partial x_{N}^{2}}$. However, finding similar
formulas for $\nabla^{\prime}\left(  s_{\lambda/\mu}\right)  $ appears
significantly harder, as the \textquotedblleft locality\textquotedblright%
\ (the fact that $\lambda$ and $\mu$ change only a very little) disappears:
For instance, for $N=3$, the expansion of $\nabla^{\prime}\left(  s_{\left(
5,3,0\right)  }\right)  $ in the Schur basis contains a $2s_{\left(
2,2,2\right)  }$ term.

\bigskip

\textbf{Acknowledgments.} The present note is a side product of the
\href{https://math.mit.edu/research/highschool/primes/YuliasDream/}{Yulia's
Dream program} 2023, in which three of its authors participated as student
researchers and the remaining one as a mentor.

\end{document}